\documentclass{article}
\usepackage{amsmath}
\usepackage{bm}
  \usepackage{paralist}
  \usepackage{graphics} %% add this and next lines if pictures should be in esp format
  \usepackage{epsfig} %For pictures: screened artwork should be set up with an 85 or 100 line screen
 \usepackage[colorlinks=true]{hyperref}
\usepackage{amsmath}
\usepackage{amsfonts}
\usepackage{amsthm}
\usepackage[]{algorithm2e}  % Warning: when you first run your tex file, some errors might occur,
\hypersetup{urlcolor=blue, citecolor=red}
\usepackage{subfigure}

\newtheorem{theorem}{Theorem}[section]

\newtheorem{lemma}[theorem]{Lemma}

\theoremstyle{definition}

\newcounter{note}
\stepcounter{note}

\usepackage{soul}

\title{Sparse approximations of  fractional Mat\'ern fields}

\begin{document}
\maketitle

% Enter the first author's name and address:
\centerline{\scshape Lassi Roininen,  Sari Lasanen, Mikko Orisp\"a\"a }
\medskip
{\footnotesize
% please put the address of the first author
\centerline{University of Oulu, Sodankyl\"a Geophysical Observatory}
 \centerline{T\"ahtel\"antie 62}
   \centerline{ FI-99600 Sodankyl\"a, FINLAND}
} % Do not forget to end the {\footnotesize by the sign }

\medskip

\centerline{\scshape Simo S\"arkk\"a}
\medskip
{\footnotesize
 % please put the address of the second  and third author
 \centerline{ Aalto University, Department of Biomedical Engineering and Computational Science}
   \centerline{P.O. Box 12200}
   \centerline{FI-00076 AALTO, FINLAND}
}

\bigskip

% The name of the associate editor will be entered by an editorial staff
% "Communicated by the associate editor name" is not needed for special issue.
 \centerline{(Communicated by the associate editor name)}

\renewcommand{\baselinestretch}{1.2}

%The abstract of your paper
\begin{abstract}
We consider a fast approximation method for a solution of a certain stochastic non-local pseudodifferential equation.
This equation defines a Mat\'ern class random field.
The approximation method is based on the spectral compactness of the solution.
We approximate the pseudodifferential operator with a Taylor expansion.
By truncating the expansion, we can construct an approximation with Gaussian Markov random fields.
We show that the solution of the truncated version can be constructed with an over-determined system of stochastic matrix equations with sparse matrices.
We solve the system of equations with a sparse Cholesky decomposition.
We consider the convergence of the discrete approximation of the solution to the continuous one.
Finally numerical examples are given.
\end{abstract}

\sloppy

\section{Introduction}
\label{intro}

We are interested in studying  generalised Gaussian Markov random fields on $\mathbb{R}^d$. 
A typical -- and often studied -- example of a Gaussian Markov random field is the Mat\'ern  field with the covariance function 
\begin{equation}\label{eqn:ACF}
C_{\mathcal X}(x,y) = \frac{  2^{1-\alpha+d/2}   \sigma^2 \ell^{2(\alpha-d/2)}}{(4\pi )^{d/2}  \Gamma(\alpha)}\left(\frac{|x-y|}{\ell^d}\right)^{\alpha-d/2}
K_{\alpha-d/2}\left(\frac{|x-y|}{\ell^d}\right),\quad x,y\in\mathbb{R}^d,
\end{equation}
where $\alpha-d/2>0$, $\sigma^2>0$ and $\ell>0$ are smoothness parameter, scaling factor, and correlation length, respectively. $K_{\alpha-d/2}$ is the modified Bessel function of the second kind and $\Gamma$ is the gamma function.
Studying the generalised Mat\'ern field is equivalent to the study of the weak solution of the stochastic partial differential equation
\begin{equation}\label{yksi}
\left(-\Delta + \kappa^2  \right)^\frac{\alpha}{2} \mathcal {X}  = \mathcal {W},
\end{equation}
where $\kappa=\ell^{-d}$ and $\mathcal {W}$ is  white noise on $\mathbb R^d$ with a covariance operator $\sigma^2 I$, where  $I$ is the identity operator \cite{Lindgren2011,Rasmussen2006,Roininen2014,Simpson2009}.
For  integer $\alpha$, fast numerical approximations of (\ref{yksi}) are well-known, see for example Lindgren et al. 2011 \cite{Lindgren2011} or Simpson 2009 \cite{Simpson2009}.  
However, an open question is how to efficiently  approximate  Mat\'ern fields with non-integer $\alpha$. % \simohox{There is something related to non-integer $\alpha$ in the Finn's paper's discussion which could be mentioned here. However, the present approach is different.}
Our work contributes to this area. 
The case of non-integer alpha was also briefly considered on page 493 of the discussion part of Lindgren et al. \cite{Lindgren2011}, where the proposed approximation is based on minimising an error functional in spectral domain. 
Although that approach, in principle, contains the Taylor series expansion as a special case, our approach differs both in the used discretisation method as well as in the respect that we formally show when the Taylor approximation lead to a valid non-degenerate covariance function.

Instead of  the continuous Mat\'ern fields, we focus on a band-limited version of the Mat\'ern fields, that is, we make a spectral truncation.
In order to make the spectral truncation, we replace the white noise $\mathcal W$ in Equation \eqref{yksi} with a spectrally truncated noise  $W$, which has a covariance function 
\begin{equation*}
C_{W }(x,y)= \frac{\sigma^2}{(2\pi)^d}\int _{|\xi|\leq \kappa}  \exp\left(-i(x-y)\cdot \xi \right) d\xi.
\end{equation*}
The corresponding stochastic partial differential equation is then
\begin{equation}\label{kaksi}
\left(-\Delta +\kappa^2\right)^\frac{\alpha}{2} X  = W,
\end{equation}
where $\alpha$ is non-integer. 
We call the solution $X$  of \eqref{kaksi} \textit{band-limited fractional Mat\'ern field}, because the covariance function of $X$ is  
\begin{equation} \label{eqn:spectral_truncated}
C_{X }(x,y)= \frac{\sigma^2}{(2\pi)^d}\int _{|\xi|\leq \kappa}  \frac{\exp\left(-i(x-y)\cdot \xi \right)}{(\kappa^2 +|\xi|^2)^{\alpha}}  d\xi.
\end{equation}
We choose the spectral truncation rule from the radius of convergence of power series expansion for the Fourier transformed  operator
$$
\mathcal F (  -\Delta +\kappa^2)^{-\alpha} (\xi), 
$$ 
which,  we  later verify, is $\kappa$.
This truncation makes the  sample paths   infinitely smooth with probability one, as we  will  also show later.

For practical computations, we may use any discretisation scheme, such as finite differences \cite{Roininen2011} or finite element methods \cite{Roininen2014}. 
From now on, we use finite differences as our discretisation scheme, because the discretised formulas are simpler than the ones obtained via finite element methods. 
In the case of the integer $\alpha$, the finite difference approximation of the Equation \eqref{yksi} leads to a sparse matrix presentation of the corresponding Mat\'ern field \cite{Lindgren2011,Roininen2014}.
The approximation is typically written as a linear stochastic matrix equation
\begin{equation} \label{eqn:prior_basicmodel}
 \mathbf{ L X} = \mathbf{W},
\end{equation}
where $\mathbf{L}$ is a sparse matrix approximating the linear operator in Equation \eqref{yksi} and $\mathbf{W}$ is discrete white noise.
The covariance of the discrete random field $\mathbf{X}$ is then 
\begin{equation} \label{eqn:covariance}
\mathbf{C} =   \left(\mathbf{ L}^T \mathbf{ L}\right)^{-1}.
\end{equation}
We note that the covariance matrix $\mathbf{C}$ is  a full matrix when $\alpha>0$, while the  precision matrix $\mathbf{ L}^T \mathbf{ L}$ is a sparse matrix.
Hence, it is appealing from the computational point of view to work with formulations of type \eqref{eqn:prior_basicmodel} rather than with the full covariance matrices.
For non-integer $\alpha$ and fractional-order difference approximations \cite{Lasanen2005,Lubich1986}, the matrix $\mathbf{L}$ is  a full matrix, hence computational efficiency is lost.
Thus, the question is raised of whether it is possible to find fast approximation, which is close enough to the original $\mathbf{L}$.
This paper aims to address this question. %This is the main topic of this paper.
We aim to do this by studying the approximations of  certain random fields closely related to Mat\'ern fields  with power spectrum defined by truncated Taylor expansions and their numerical approximations.

Our main motivation for studying band-limited fractional Mat\'ern fields is in applying them as prior distributions in Bayesian statistical inverse problems \cite{Kaipio2005}. 
In our earlier studies, we have considered Gaussian Markov random fields within the framework of Bayesian statistical inverse problems (Roininen et al. 2011 and 2013 \cite{Roininen2011,Roininen2013a}) and applied the methodology to an electrical impedance tomography problem (Roininen et al. 2014 \cite{Roininen2014}). 
Studies of very high dimensional prior distributions arising from spatially sampled values  of random fields in Bayesian inversion are reported by Lasanen 2012  \cite{Lasanen2012a} and Stuart 2010 \cite{Stuart2010}. 
In S\"arkk\"a et al. 2013 \cite{Sarkka2013} and Solin et al. 2013 \cite{Solin2013} we also applied Mat\'ern and other types of spatio-temporal Gaussian random fields to fMRI brain imaging and prediction of local precipitation, and in Hiltunen et al. 2011 \cite{Hiltunen2011} to diffuse optical tomography.
Other applications of Mat\'ern fields include for example spatial interpolation \cite{Lindgren2011} and machine learning \cite{Rasmussen2006}.

This paper is organised as follows:
In Section \ref{sec:corrpri}, we consider the  approximation of the fractional spectrum with truncated Taylor expansion and discuss corresponding discrete approximations with sparse matrices.
In Section \ref{sec:QR}, we construct  upper triangular matrix  $\mathbf{L}$ (see Equation \eqref{eqn:prior_basicmodel}) with Cholesky decomposition. 
In Section \ref{sec:bl}, we further consider  Taylor expansion of power spectrum in more detail.
The convergence of the discrete approximations to the continuous ones will be considered in Section \ref{sec:convergence}.
Finally, in Section \ref{sec:experiments}, we numerically study the accuracy of the approximation in the case of 2-dimensional Mat\'ern field.

\section{Approximating band-limited covariances}
\label{sec:corrpri}

Our aim, in this section, is to  study approximations of band-limited Mat\'ern fields \eqref{eqn:spectral_truncated} in two steps: First we approximate the fractional spectrum $(\kappa^2 +|\xi|^2)^{\alpha}$ with truncated Taylor series. 
Then we study discrete approximations of the corresponding random field via trigonometric polynomials, which lead to matrix covariance formulas of type \eqref{eqn:covariance}.

Let us denote
\begin{equation} \label{eqn:Pt}
P(t):= \left(\kappa^2+ t ^2\right)^{\alpha}, 
\end{equation} 
where $t\in\mathbb R$. 
The function $P$ has the well-known Taylor series 
\begin{equation}\label{tayl}
\left(\kappa^2+ t ^2\right)^\alpha = \sum_{k=0}^\infty a_k \kappa^{2 (\alpha-k)} t ^{2k},
\end{equation}
where 
\begin{equation}\label{eq:coefficients}
\begin{split}
a_0&=1, \\
a_k&= \frac{\alpha (\alpha -1)\cdots (\alpha -k+1)}{k!}\quad\text{for}~~k\geq1.
\end{split}
\end{equation}
We note that the series \eqref{tayl} converges for $ |t| \leq \kappa $ and   diverges  for $|t|>\kappa$. 
In Section \ref{sec:bl}, we verify that the divergence is due to unlimitedness of the partial sums. 

We apply the Taylor series \eqref{tayl}  to the covariance function in Equation \eqref{eqn:spectral_truncated} and obtain
\begin{equation}
\begin{split}
C_{X }(x,y)&= \frac{\sigma^2}{(2\pi)^d}\int _{|\xi|\leq \kappa}  \frac{\exp\left(-i\left(x-y\right)\cdot \xi \right)}{ P(|\xi|)}  d\xi  \\
&= \frac{\sigma^2}{(2\pi)^d}\int_{|\xi|\leq \kappa}   \frac{ \exp\left(-i\left(x-y\right)\cdot\xi \right)}{ \sum_{k=0}^\infty a_k \kappa^{2 (\alpha-k)} |\xi| ^{2k}  } d\xi.  \label{series}
\end{split}
\end{equation}
As our objective is to find a Gaussian Markov random field approximation, we  truncate the Taylor series in \eqref{series}, and set
\begin{equation} \label{eqn:taylor_trunc}
C_{X}^{K} (x,y) = \frac{\sigma^2}{(2\pi)^d} \int    \frac{ \exp\left(-i\left(x-y\right)\cdot\xi \right)}{ \sum_{k=0}^K a_k \kappa^{2 (\alpha-k)} |\xi| ^{2k}  } d\xi. 
\end{equation}
We choose the truncation level $K$ in such a way that $a_K>0$. 
This guarantees the positivity of the denominator (see Section \ref{sec:bl} for detailed discussion).

The band-limited spectral density in Equation \eqref{eqn:spectral_truncated} as such can result in quite large differences to the covariance function due to the missing tails. 
However, it often turns out that the truncated series in Equation \eqref{eqn:taylor_trunc} is positive in a considerably larger area $|\xi|\leq \kappa'$ with $\kappa' > \kappa$ even though the Taylor series converges only in $|\xi|\leq \kappa$. 
Quite often it is even valid in the whole $\mathbb{R}^d$.
In those cases, by extending the integration area as done in Equation \eqref{eqn:taylor_trunc}, we can better retain the tails of the spectral density which leads
to a considerably more accurate approximation to the covariance function.

As an example, we choose $d=1$, $\sigma^2 = 1$, $\alpha=3/2$ and $\kappa=1$ with truncation parameter $K=4$, and set
\begin{equation} \label{eqn:taylor_approx}
P(t) =1+ \frac{3}{2}t^2   + \frac{3}{8}t^4 - \frac{1}{16} t^6 + \frac{3}{128} t^8.
\end{equation}
The polynomial is clearly everywhere positive and hence the spectral density is valid in the whole $\mathbb{R}$.
Thus we can extend the integration area to the whole space. 
Figure~\ref{fig:trunc_1d} illustrates the resulting approximation.
The general case is studied in Lemma \eqref{raj} (Section \ref{sec:bl}).

\begin{figure}[htp]
\begin{center}
  % replace aims_logo.eps by your figure file name
%  \includegraphics[width=0.6\textwidth]{measurement_placeholder}\\
  \subfigure[Plain band-limited]{\includegraphics[width=0.49\textwidth]{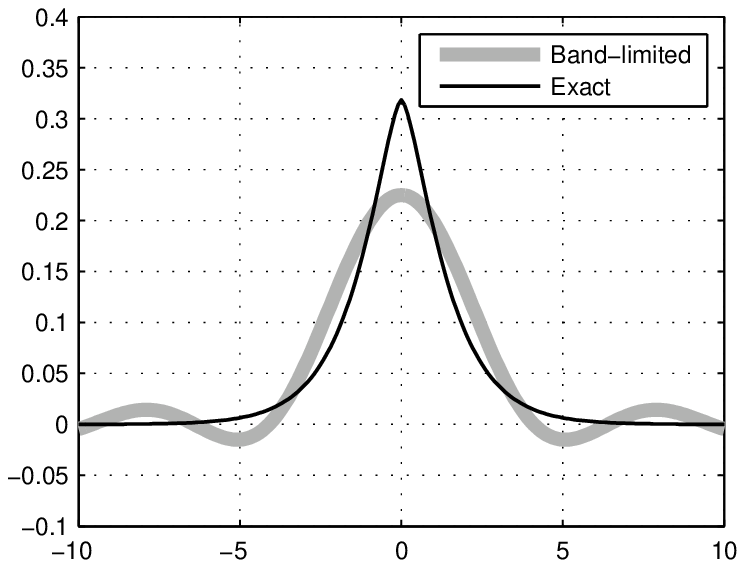}}
  \subfigure[Taylor expansion]{\includegraphics[width=0.49\textwidth]{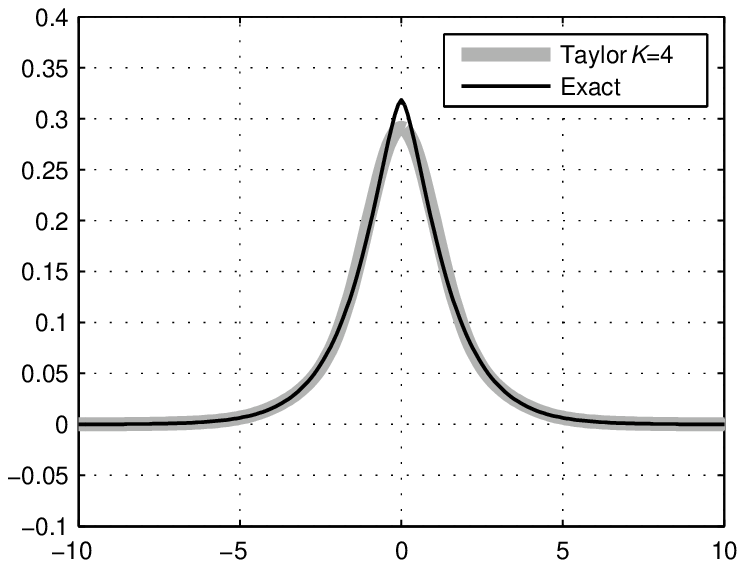}}
  \caption{(a) Covariance function of a band-limited approximation to one-dimensional Mat\'ern spectral density. (b) Covariance function of a $4^{(\mathrm{th})}$ order Taylor series expansion. Although the plain band-limited approximation is quite inaccurate, the truncated Taylor series on the whole $\mathbb{R}$ is quite accurate.
  } \label{fig:trunc_1d}
  \end{center}
\end{figure}

We give the discrete approximation on lattice $h(\mathbf {i},\mathbf {j})^d$, where $h>0$ is discretisation step and $\mathbf {i},\mathbf {j}\in\mathbb{Z}^d$.
Then the discrete approximation of the continuous covariance \eqref{eqn:taylor_trunc} can be written with the discrete Fourier transform and trigonometric polynomials \cite{Roininen2011} as
\begin{equation} \label{eqn:approx_disc_d}
C(\mathbf {i},\mathbf {j}) = \frac{\sigma^2}{(2\pi)^d} \int_{(-\pi,\pi)^d} \frac{\exp\left(-i(\mathbf {i}-\mathbf {j})\cdot\xi\right)}{\sum_{k=0}^Ka_k \kappa^{2 (\alpha-k)}h^{d-2k}  \left(\sum_{p=1}^d(2-2\cos(\xi_p))^{k}\right)}d\xi.
\end{equation}
We emphasise that the integrand is not band-limited to $\vert \xi \vert \leq \kappa$, because of the approximations applied.

For the relationship between polynomials corresponding to continuous covariance \eqref{eqn:taylor_trunc} and trigonometric polynomials corresponding to discrete covariance \eqref{eqn:approx_disc_d}, see Section \ref{sec:convergence}.
The difference between our earlier study \cite{Roininen2011} and this paper, is that here we let the terms  $c_k:=a_k \kappa^{2 (\alpha-k)}$ to be also negative.
However, as mentioned earlier, we require that the sum is strictly positive, as is the case in the formulation in Equation \eqref{eqn:taylor_approx}.
In Section \ref{sec:QR}, we will consider a technique for constructing matrix $\mathbf{L}$ in Equation \eqref{eqn:prior_basicmodel}  for the cases $c_k\in \mathbb{R}$.

Studying  the trigonometric polynomial
\begin{equation}
2-2\cos(\xi_p) = |1-\exp(i\xi_p)|^2
\end{equation}
in Equation \eqref{eqn:approx_disc_d} is related to the study of difference matrices \cite{Roininen2011}.
For example, let us choose $d=1$.
Then  we can write a stochastic first order difference matrix equation  as
\begin{equation} \label{eqn:L1_matrix}
\mathbf{L}_1\mathbf{X} = \frac{ \delta_0 - \delta_{-1}}{h} * \mathbf{X} = \mathbf{W}  \Leftrightarrow  \frac{1-\exp(i\xi)}{h} \mathcal{F}(\mathbf{X}) = \mathcal{F}(\mathbf{W}),
\end{equation}
where $\delta$ is the Kronecker delta, white noise $\mathbf{W}$ has covariance  $\Sigma_1=\frac{1}{hc_1}\mathbf{I}$.
We can write  $k^{(\mathrm{th})}$-order difference matrix $\mathbf{L}_k$ similarly to Equation \eqref{eqn:L1_matrix}.
The corresponding covariance matrices $\Sigma_k$ are obtained from the constants $c_kh^{d-2k}$ in Equation \eqref{eqn:approx_disc_d} and they are 
$\Sigma_k=\frac{h^{2k-1}}{c_k}\mathbf{I}$. 
Using the additivity property of the precision matrix \cite{Roininen2011}, we may then write the discrete covariance with matrix equations as 
\begin{equation} \label{eqn:covarianssi_koko}
\mathbf{C}  = \left(\mathbf{L}^T\mathbf{L}\right)^{-1}=\left( \sum_{k=0}^K \mathbf{L}_k^T \Sigma_k^{-1} \mathbf{L}_k\right)^{-1} = \left( \sum_{k=0}^K a_k \kappa^{2 (\alpha-k)} h^{1-2k}\mathbf{L}_k^T \mathbf{L}_k\right)^{-1}. 
\end{equation}

\section{Cholesky decomposition}
\label{sec:QR}

Given the matrices $\mathbf{L}_k$ and $\Sigma_K$ in Equation \eqref{eqn:covarianssi_koko}, our aim is to construct an upper triangular sparse matrix $\mathbf{L}$.
The full covariance matrix $\mathbf{C}$ in Equation \eqref{eqn:covarianssi_koko}, has both $c_k>0$ and $c_k<0$ terms, which we relate to constructing $\mathbf{L}$ with Cholesky  decomposition. 
We choose this construction, as we aim to construct the upper triangular matrix $\mathbf{L}$ term-by-term,  that is, we recursively  apply 
the Cholesky decomposition in order to get the wanted presentation.
Cholesky decomposition  
algorithms are covered in standard literature \cite{Golub1996}, and they are extensively used in square root Kalman filtering (see, e.g., \cite{Bierman1977,Grewal2008}).

An alternative to Cholesky decomposition is the QR decomposition with Givens rotations and anti-rotations \cite{Orispaa2010}.
We note that mathematically, from the perspective of this paper, Cholesky and QR methods are equivalent.

In the computation of estimators to inverse problems with prior covariance $\mathbf{C}$ as well as in simulation of the random field we are interested in performing matrix-vector operators of the form $\mathbf{L}^{-1} \, \mathbf{v}$, where $\mathbf{v}$ is some given vector. 
When $\mathbf{L}$ is a sparse matrix, this can be efficiently evaluated without explicitly computing the (full) matrix inverse $\mathbf{L}^{-1}$. 
Although the matrix $\mathbf{L}$ can be computed via factoring $\mathbf{C}^{-1}$, for maximal numerical accuracy it beneficial to compute it directly without computing $\mathbf{C}^{-1}$. 
This is because the number of bits required for a given floating point precision for constructing $\mathbf{C}^{-1}$ is twice the required bits for $\mathbf{L}$ \cite{Bierman1977}.

We start by partitioning the precision matrix  $\mathbf{C}^{-1}$ as 
\begin{equation}
\begin{split}
 \mathbf{P}_{+} - \mathbf{P}_{-} :=  \sum_{a_k\geq 0}  a_k \kappa^{2 (\alpha-k)} h^{1-2k}\mathbf{L}_k^T \mathbf{L}_k - \sum_{a_k< 0}  \vert a_k \kappa^{2 (\alpha-k)} \vert h^{1-2k}\mathbf{L}_k^T \mathbf{L}_k, % = \mathbf{L}^T\mathbf{L} ,
\end{split}
\end{equation}
where the partitioned precision matrices $\mathbf{P}_{+}$ and $\mathbf{P}_{-}$ correspond to the parts to be  sequentially updated with positive and negative signs, respectively.
When making the Cholesky decomposition,  we first loop over the positive $c_k$ coefficients and do Cholesky updates with $\sqrt{a_k \kappa^{2 (\alpha-k)} h^{1-2k}} \, \mathbf{L}_k$. 
Then we do the same for the negative coefficients with the so-called Cholesky downdates.
%There are several update algorithms, the important thing is that we update only the matrix $\mathbf{L}$ and we do not have to work on the square-forms $\mathbf{L}^T\mathbf{L}$ or its inverse.
%%
%For update algorithms, see, for example, \cite{Golub1996}.
%%
%After the update algorithm, we obtain the wanted result.
%%
We note that it advisable not to mix the updates with positive  and negative signs, because this might break the positive-definiteness  property of the covariance matrix. 
This might break the algorithm and hence, we propose to carry out updates with positive signs first and downdates with negative signs at the last part of the algorithm.

A further development of the matrix factorisation is to use the $\mathbf{L}^T\mathbf{DL}$ decomposition, where $\mathbf{D}$ is a diagonal matrix and the diagonal elements of the $\mathbf{L}$ are all ones. 
Hence, this allows the presentation of the form
\begin{equation}
\mathbf{LX} = \mathbf{W} \sim \mathcal{N}(0,\mathbf{D}^{-1}).
\end{equation}
The inverse $\mathbf{D}^{-1}$ is fast to compute as it is a diagonal matrix.
Figure~\ref{fig:chol_1d} shows an example of a covariance function approximation formed with the above procedure (using the SuiteSparse\footnote{For SuiteSparse software package, see http://faculty.cse.tamu.edu/davis/suitesparse.html.} library \cite{Davis2011,Foster2013}) as well as example realisations of the process.

\begin{figure}[htp]
\begin{center}
  % replace aims_logo.eps by your figure file name
%  \includegraphics[width=0.6\textwidth]{measurement_placeholder}\\
  \subfigure[Approximation]{\includegraphics[width=0.49\textwidth]{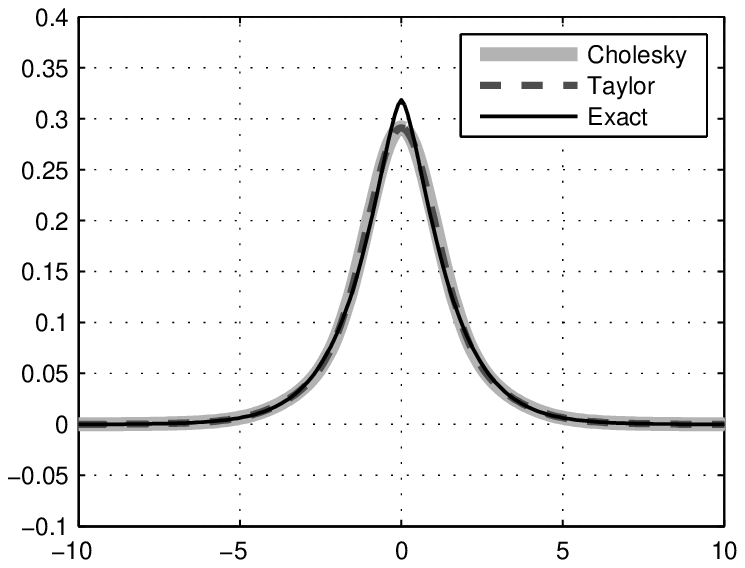}}
  \subfigure[Realisations]{\includegraphics[width=0.49\textwidth]{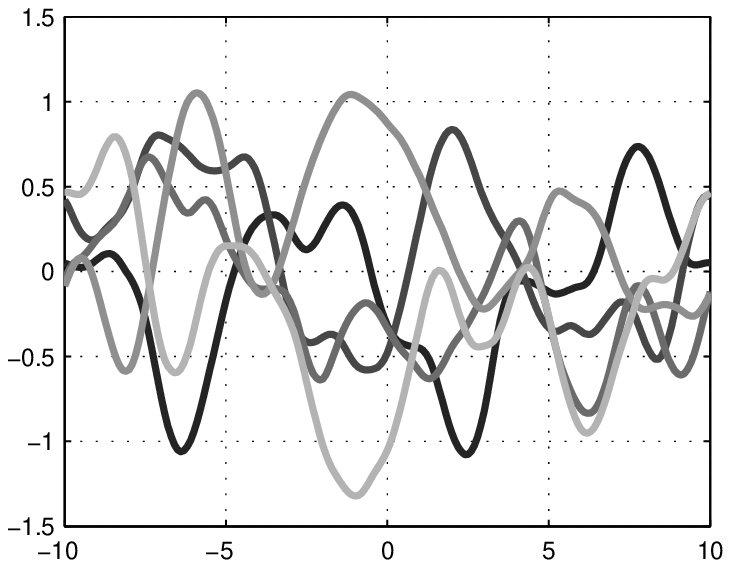}}
  \caption{(a) Exact covariance function of the one-dimensional example of Equation \eqref{eqn:taylor_approx} ($\sigma^2 = 1$, $\alpha=3/2$, $\kappa=1$, $K=4$), the truncated Taylor series approximation and its finite-difference approximation with discretisation step $h = 0.1$. In the finite-difference computations, we have used periodic boundary conditions in an extended domain and cropped the image. (b) Realisations from the process simulated via the discretised approximation.
  } \label{fig:chol_1d}
  \end{center}
\end{figure}

\section{Band-limited fractional Mat\'ern fields}
\label{sec:bl}

In this section, we  discuss the power spectrum of the band-limited Mat\'ern fields with certain expansion schemes.
We first note a fundamental property of the band-limited Mat\'ern fields:
\begin{lemma}\label{smooth}
Let $X$ be the solution of \eqref{kaksi}. Then 
the sample paths of $X$ are smooth on $\mathbb R^d$ with 
probability 1.
\end{lemma}
\begin{proof}
For any $k\in\mathbb N$ and $j=1,\cdots, d$, the spectra $$S(\xi)= \sigma^2   \frac{|\xi_j|^{2k} }{(\kappa^2+ |\xi|^2)^\alpha }1_{|\xi|\leq \kappa} (\xi) $$ of the weak derivatives  $\frac{\partial^k  X}{\partial x_j^k}$ satisfy the condition  
 \begin{equation}
\int S(\xi) (\log(1+|\xi|))^{1+\epsilon} d\xi <\infty
 \end{equation}
 for  fixed  $\epsilon>0$.   By  Theorem 3.4.3 in \cite{Adler1981}, 
the weak derivatives  are almost surely continuous. 
\end{proof}

\subsection{Convergence of the truncated Taylor series} 

In order to study the convergence of the truncated Taylor approximations \eqref{eqn:taylor_trunc}, we need a preliminary result:
\begin{lemma}\label{raj}
Let    $K=[\alpha]+2J+1$, where  integer $J\geq1$ and
 $[\alpha]$ is the integer part of $\alpha> \frac{d}{2}$.  Let $\kappa>0$ and 
coefficients $a_k$  be  as in \eqref{eq:coefficients}.
Then there is a constant $c>0$ independent of $J$ such that 
\begin{equation}\label{partial}
c(1 +|\xi|^{d+2})  \leq \sum_{k=0}^K a_k \kappa^{2(\alpha- k)} |\xi| ^{2k}  
\end{equation}
for all $ \xi \in\mathbb R^d$. Moreover,  there are non-negative polynomials 
$q_j$,  $j=1,\dots,J$, such that 
\begin{equation}
\sum_{j=1}^J q_j(|\xi|)  \leq \sum_{k=0}^K a_k \kappa^{2(\alpha- k)} |\xi| ^{2k},  
\end{equation}
for every $K$, and 
\begin{equation}\label{eq:limit}
\lim_{J\rightarrow \infty}   \sum_{j=1}^J q_j(t) = +\infty
\end{equation}
for $t>1$. 
\end{lemma}  
\begin{proof}
We first set 
$$
 q_0(t):= \left( \sum_{k=0}^{[\alpha]} a_k t^{2k}\right) +\frac{1}{2}  a_{[\alpha]+1} t^{2([\alpha]+1)}$$%
for all $t\geq 0$.
By  positivity of coefficients $a_k$, $k=0,\dots,[\alpha]+1$ (see  \eqref{eq:coefficients}), we have
\begin{equation}\label{eq:lower_bound}
q_0(t) \geq c(1+t ^{d+2 }) 
\end{equation}
for $t\geq 0$. 
For $j=1,\dots,J$, we set 
\begin{equation*}
\begin{split}
q_j(t):=\frac{1}{2} a_{[\alpha]+2j-1} t^{2([\alpha]+2j-1)} + a_{[\alpha]+2j} t^{2([\alpha]+2j)} +\frac{1}{2} a_{[\alpha]+2j+1} t^{2([\alpha]+2j+1)} \end{split}.
\end{equation*}
The sum  in  \eqref{partial} has then the  expression 
$$
 \sum_{k=0}^K a_k \kappa^{2(\alpha-k)} |\xi| ^{2k}   = \kappa ^{2\alpha }  \left(  \sum_{j=0}^{J}q_j(\kappa^{-1}|\xi| )\right)
 +  \frac{1}{2}  a_{[\alpha]+2J+1} \kappa^{2\alpha} (\kappa^{-1}|\xi|)^{2([\alpha]+2J+1)},
$$
where also $ a_{[\alpha]+2J+1} >0$.
 From \eqref{eq:coefficients}, we observe that  for $j=1,\dots,J$
\begin{multline}\label{eq:for_limit}
q_j(t) =     \frac{1}{2} a_{[\alpha]+2j-1} t^{2([\alpha]+2j-1)}  \\
 \times \left(1+ \frac{2(\alpha - [\alpha]-2j+1) t^2}{ [\alpha]+2j} +
\frac{(\alpha-[\alpha] -2j ) (\alpha -[\alpha]-2j+1)t^4}{ ([\alpha] + 2j)([\alpha]+2j+1)} \right).
\end{multline}
The only real zero of $q_j$, $j=1,\dots,J$,  is at zero, since
the discriminant for the quadratic factor  in $q_j(\sqrt t)$  is negative. Indeed,  
\begin{equation} 
\begin{split}
D &= 4 \left(\frac{\alpha- [\alpha]- 2j+1}{[\alpha]+2j}  \right)^2 - 4 \frac{(\alpha- [\alpha]- 2j))(\alpha- [\alpha]- 2j +1)}{
([\alpha]+2j)([\alpha]+2j+1)}\\
%&= \frac{4 (\alpha-  [\alpha]- 2j  +1)}{[\alpha]+2j}  \left(\frac{\alpha-  [\alpha]- 2j  +1}{[\alpha]+2j} - \frac{\alpha- [\alpha]-2j  }{[\alpha]+2j+1}  \right)\\
%&=  \frac{4 (\alpha-  [\alpha]- 2j  +1)}{([\alpha]+2j)^2([\alpha]+2j+1)}  \left((\alpha-  [\alpha] - 2j +1)([\alpha]+2j+1) -  (\alpha- [\alpha]- 2j ) ([\alpha]+2j)  \right)\\
&=  \frac{4 (\alpha- [\alpha]- 2j) +1)}{([\alpha]+2j)^2([\alpha]+2j+1)}  \left(\alpha +1  \right),
\end{split}
\end{equation}
where $\alpha-[\alpha]-2j + 1<0$ for $j\geq 1$. Moreover,  we see that $q_j\geq 0$ by inspecting the signs of the coefficients of 
$q_k$.
For the limit \eqref{eq:limit}, we note that by \eqref{eq:for_limit}, 
the ratio of the consecutive terms has the limit
\begin{equation}
\begin{split}
\lim_{j \rightarrow \infty } \frac{q_{j+1}(t)}{q_{j}(t)}  &= 
\lim_{j\rightarrow \infty} \frac{a_{[\alpha]+2j+1}}{a_{[\alpha]+2j-1}} t^2 
\frac{ \left(1+ \frac{2(\alpha - [\alpha]-2j-1) t^2}{ [\alpha]+2j-2} +
\frac{(\alpha-[\alpha] -2j-2 ) (\alpha -[\alpha]-2j-1)t^4}{ ([\alpha] + 2j-2)([\alpha]+2j-1)} \right) 
}{ \left(1+ \frac{2(\alpha - [\alpha]-2j+1) t^2}{ [\alpha]+2j} +
\frac{(\alpha-[\alpha] -2j ) (\alpha -[\alpha]-2j+1)t^4}{ ([\alpha] + 2j)([\alpha]+2j+1)} \right) 
}\\
&= t^2,
\end{split}
 \end{equation}
 which shows that the corresponding series diverges for $t>1$. Since the series is a sum of non-negative terms,
 it is unbounded.
\end{proof}

When $K=[\alpha]+2J+1$,   the truncated Taylor approximation 
$$
C_X^{K} (x,y) =  \frac{\sigma^2}{(2\pi)^d } \int   \frac{\exp\left(-i\left(x-y\right)\cdot\xi \right)}{\sum_{k=0}^{K}  a_k \kappa^{2(\alpha- k)} |\xi|^{2k}}d\xi,
$$ 
is a well-defined  function by \eqref{partial}. Moreover, the truncated Taylor approximation
$C_X^{K}$ satisfies  the equation
\begin{equation}\label{eq:trunc_cov}
 \sum_{k=0}^{K}  a_k \kappa^{2(\alpha- k)}\Delta ^{k}_x  C_X^{K} (x,y) = - \delta_{y}
\end{equation}
in the sense of tempered  distributions. 
Taking the Fourier transform and dividing by the positive term $\sum_{k=0}^{K}  a_k \kappa^{2(\alpha- k)}|\xi|  ^{k}$ 
leads to the equation.

We proceed  to  study convergence of covariances when  using  truncated Taylor approximations.
The following theorem demonstrates that   restrictions on the spectral domain are not required 
when using the Taylor  approximations. This is a significant benefit for the numerical approach 
in terms of computational speed.
 \begin{theorem}
 Let $K_J=[\alpha]+2J+1$ for $J\in\mathbb N$. 
The approximations 
$$
C_X^{K_J} (x,y) =  \frac{\sigma^2}{(2\pi)^d } \int   \frac{\exp\left(-i\left(x-y\right)\cdot\xi \right)}{\sum_{k=0}^{K_J}  a_k \kappa^{2(\alpha- k)} |\xi|^{2k}}d\xi
$$
converge uniformly to 
$$
C_X(x,y)= \frac{\sigma^2}{(2\pi)^d } \int_{|\xi|\leq \kappa}   \frac{\exp\left(-i\left(x-y\right)\cdot\xi \right)}{(\kappa^2+ |\xi|^2 )^\alpha}d\xi
$$
as $J\rightarrow \infty$. 
\end{theorem}
\begin{proof}
For simplicity, take $\kappa=\sigma=1$.  Denote  $K=[\alpha]+2J+1$ and set
\begin{equation}\label{eq:estimate}
\begin{split}
e_{K}:=&\sup_{x,y}  \left| (2\pi)^d  C_X (x,y)-  \int   \frac{\exp\left(-i\left(x-y\right)\cdot\xi \right)}{\sum_{k=0}^K  a_k   |\xi|^{2k}}d\xi  \right | \\
\leq &   \int_{|\xi|\leq  1 }    \left| 
\frac{1}{\sum_{k=1}^\infty   a_k   |\xi|^{2k}  }-\frac{1}{\sum_{k=1}^K   a_k   |\xi|^{2k}  } \right| d\xi  +  
\int _{|\xi|\geq 1  }  \left|   \frac{1}{\sum_{k=0}^K  a_k  |\xi|^{2k}} \right| d\xi \\
= &  \int_{|\xi|\leq  1 }    \left| 
\frac{ \sum_{k=K+1}^\infty   a_k   |\xi|^{2k}  }{\left(  \sum_{k=1}^{\infty}   a_k   |\xi|^{2k}\right)\left( \sum_{k=1}^K   a_k   |\xi|^{2k} \right) } \right| d\xi  + 
\int _{|\xi|\geq 1  }  \left|   \frac{1}{\sum_{k=0}^K  a_k  |\xi|^{2k}} \right| d\xi. \\
\end{split}
\end{equation}

The remainder term of the Taylor approximation for $(1+t ^2)^\alpha$   at $t=|\xi|\leq 1$  is 
\begin{equation}
\begin{split}
R_K (|\xi|^2)\leq&  \int_ 0^{|\xi|^2}  
\frac{|\alpha(\alpha-1)\dots  (\alpha-K)|  }{K!} \times     (|\xi|  ^2 - s )^K (1+s)^{\alpha-K-1} ds\\
\leq&  \left|\frac{ |\alpha(\alpha-1)\dots  (\alpha-K+1)|  }{K!  }  (2^{\alpha-K}-1)\right|
\end{split}
\end{equation}
which converges uniformly.  
 Together with the lower bound \eqref{partial}, this bounds the first integral
  in \eqref{eq:estimate}.

For the second integral, we use from Lemma \ref{raj}, the lower bound 
\begin{equation}\label{eq:rhs}
\sum_{k=0}^{[\alpha]+2J+1 } a_k  |\xi|   ^{2k} \geq \sum_{j=1}^J  q_j(|\xi| ), 
\end{equation}
and the limit \eqref{eq:limit}.  Hence,  the second integral vanishes 
as $J\rightarrow \infty$.
\end{proof}

\section{Convergence of the discrete field to continuous}
\label{sec:convergence}

In this section, we define the lattice approximations of the random field $X$ similarly as in 
\cite{Roininen2011}. 
We choose the discrete lattice  to be  $h\mathbb{Z}^d$, where $h>0$.
The continuous field $X$ is first  restricted onto $h\mathbb {Z}^d$ and then approximated by a
discrete field $X^{(h)}$ on $h\mathbb {Z}^d$.
 
We start by discretising  the Laplacian in \eqref{eq:trunc_cov}. 
As the discretisation scheme, we use finite-difference methods. 
That is,  for $d=1$ the  discrete Laplacian is
$$
(\Delta_h  f)(hi) = \frac{f(h(i-1)) -2 f(hi)+ f(h(i-1))}{h^2},  
$$
where $\; i\in \mathbb Z$ and $f: h\mathbb Z \rightarrow \mathbb R$.
Similarly, in dimension $d=2$
\begin{equation*}
\begin{split}
(\Delta_h f) (hi,hj) = & h^{-2}(- 4 f(hi,hj)  +  f(h(i-1),hj)   \\
&  + f (hi,h(j-1)) + f(hi,h(j+1)) + f(h(i+1),hj)),    %\\&&+ f(h(i+1),hj)).
\end{split}
\end{equation*}
where $i,j\in\mathbb Z$ and $f: h\mathbb Z^2 \rightarrow \mathbb R$.

We define  the lattice approximation  $X^{(h)}$ as a zero mean Gaussian random field 
on $h\mathbb Z^d$ whose stationary covariance $ C_{X^{(h)}}(h\mathbf{ i}, h\mathbf{j}))=C_{X^{(h)}}(h (\mathbf{i-j})),  \; \mathbf {i},\mathbf {j}\in \mathbb Z^d$  is given by a  discrete version of \eqref{eq:trunc_cov}:
\begin{equation}\label{covar}
  \sum_{k=0}^{K}  a_k \frac{ h^{d}}{\sigma^2}\kappa^{2(\alpha- k)}\Delta ^{k}_{h}  C_{X^{(h)}} (h\mathbf i, h\mathbf j) = -  \delta_{\mathbf{i j}}
\end{equation}
where  $\delta_{\mathbf{ij}}$ is the Kronecker delta function on $\mathbb Z^d$, and the discrete Laplacian operates on 
the first variable. The multiplier $h^d$ on the left hand side is connected to the convergence of the discretization. Namely, it  distinguishes the lattice approximations of continuous integral operators from  their kernels, which are studied   here.  In \cite{Roininen2011}, this multiplier is included in the construction of the discrete prior. 

Define the spectrum of $f:  h\mathbb Z^d\rightarrow \mathbb R$ as 
\begin{equation*}
S_f (\xi )=\sum_{\mathbf i\in \mathbb Z^d }  f(h\mathbf i) \exp(  i\, \mathbf {i} \cdot \xi),
\end{equation*}
for all $\xi \in (-\pi,\pi)^d$.  Then 
\begin{equation*}
f(h\mathbf i)= \frac{1}{(2\pi)^d} \int_{(-\pi,\pi)^d}  S_f(\xi ) \exp( -   i\, \mathbf {i}\cdot \xi ) d\xi.
 \end{equation*}
Next, we consider the discrete Laplacian as a Fourier multiplier:
\begin{equation*}
  \widehat {(- \Delta_h f)}(\xi)=   \sum_{p=1}^d \frac{2-2 \cos( \xi_p )}{h^2} \widehat f (\xi),
\end{equation*}
and transform \eqref{covar} into
\begin{equation}\label{eq:covar2}
\bigg( \frac{h^{d} }{\sigma^2}\sum_{k=0}^{K}  a_k \kappa^{2(\alpha- k)}
 \sum_{p =1}^d  h^{-2k}(2- 2\cos(\xi_p))^k   \bigg)\widehat C_{X^{(h)}} (\xi) =1.
\end{equation}
Since $2-2\cos(\xi_p)\geq 0$, Lemma \ref{raj} implies the positivity of the 
multiplier of $\widehat C_{X^{(h)}} (\xi) $ in \eqref{eq:covar2}.
The spectrum of $C_{X^{(h)}}$  is then obtained from \eqref{eq:covar2} as
\begin{equation*}
\begin{split}
\widehat C_{X^{(h)}}(\xi)=&   \frac{\sigma^2} {h^{d}} \bigg( \sum_{k=0}^{K}  a_k \kappa^{2(\alpha- k)}
 \sum_{p =1}^d  h^{-2k}(2- 2\cos(\xi_p))^k   \bigg)^{-1}.
\end{split}
\end{equation*}
The corresponding discrete correlation is
\begin{equation}\label{eq:covar3}
\begin{split}
C_{X^{(h)}}(h \mathbf {i})= \frac{1}{(2\pi)^d}\int_{(-\pi,\pi)^d}    \exp (- i \,    \mathbf i\cdot \xi )  \widehat C_{X^{(h)}}(\xi) d\xi . 
\end{split}
\end{equation}

Let us make a change of variables in the integral \eqref{eq:covar3}:
\begin{equation}\label{eq:denominator} 
C_{X^{(h)}}(h \mathbf i )= \frac{\sigma^2}{(2\pi)^d }\int_ {\left(-\frac{\pi }{h}, \frac{\pi }{h}\right)^d}  
\frac{  \exp (- i     h   \mathbf i\cdot  \xi  ) }{   
   \sum_{k=0}^{K}  a_k \kappa^{2(\alpha- k)}
 \sum_{p =1}^d  h^{-2k}(2- 2\cos(\xi_p))^k  } d\xi. 
\end{equation}
For a  given $x\in \mathbb R^d$, we  choose such a sequence of  $\mathbf i =\mathbf  {i}_{h}$ that $h\mathbf {i}_h \rightarrow  x$ as  $h\rightarrow 0$, and apply Lebesgue's dominated convergence theorem for \eqref{eq:denominator} as 
$h\rightarrow 0$.  Indeed, 
by  Jordan's  inequality,  
the lower bound 
\begin{equation*}
h^{-2}(2-2\cos(\xi_p/h))= 4 h^{-2 }\sin^2(\xi_p/2h)\geq 4 \xi_p^2 /\pi^2
\end{equation*}
holds, which  together with Lemma \ref{raj}   implies
 that the denominator in \eqref{eq:denominator} has the lower bound 
\begin{equation*}
\begin{split}
 \bigg( \sum_{k=0}^{K}  a_k \kappa^{2(\alpha- k)}
 \sum_{p =1}^d  h^{-2k}(2- 2\cos(\xi_p/h))^k   \bigg)  \geq  c  (1+  |\xi|^{2d+2} ), 
\end{split}
\end{equation*}
where $c$ does not depend on $h$. Moreover,
\begin{equation*}
\lim_{h\rightarrow 0 }  h^{-2}(2-2\cos(h\xi')) =\xi'^2.
\end{equation*}
  We have shown the following result:
\begin{theorem}
Let $K$  and $a_k$ be as in Lemma \ref{raj}. Let
$$
C_{X^{(h)}}(h \mathbf i, h\mathbf j )=\frac{\sigma^2}{(2\pi)^d}  \int_{(-\pi,\pi)^d}    \frac{\exp( - i   ( \mathbf i-\mathbf j) \cdot \xi )}{h^{d} \sum_{k=0}^{K}  a_k \kappa^{2(\alpha- k)}
 \sum_{p =1}^d  h^{-2k}(2- 2\cos(\xi_p))^k   } d\xi,  \;  
 $$
for all $\mathbf i,\mathbf j\in \mathbb Z^d
$ and 
$$
C_X^K(x,y)= \frac{\sigma^2}{(2\pi)^d } \int   \frac{\exp\left(-i\left(x-y\right)\cdot\xi \right)}{\sum_{k=0}^{K}  a_k \kappa^{2(\alpha- k)} |\xi|^{2k}}d\xi, \; 
$$
for all $x,y\in\mathbb R^d$.
Then 
$$
\lim_{h\rightarrow 0} C_{X^{(h)}} (h\mathbf {i}_h, h\mathbf {j}_h) =C_X (x,y)
$$
whenever sequences $(\mathbf i_h)$ and $(\mathbf j_h)$ are such that  $\lim_{h\rightarrow 0} (h\mathbf i_{h}, h\mathbf j_{h}) = (x,y)$.
\end{theorem}

Pointwise convergence of covariance functions shows
that corresponding finite-dimensional Gaussian distributions converge 
weakly in the sense of measures, that is,  expectations of bounded continuous functions  converge. 
In applications, more is often needed. Namely,    mappings  $H$ defined on paths of  random fields 
often appear, for example,  in Bayesian inverse problems.   This raises the need to  study convergence of 
 interpolated  random fields.  The domain of definition of the continuous mapping $H$
 usually  dictates the function space and topology in which  the convergence is to be studied. Below, we demonstrate 
 one result of this kind.

We show that  Whittaker--Shannon interpolated  random fields    converge weakly in the sense of measures on   $C(\mathbb R^d )$.   Here the space 
$C(\mathbb R^d)$ is  equipped with the usual metric topology corresponding to the family of seminorms $\vert f \vert_k :=\sup_{x\in K_k} \vert  f(x)\vert $, where the union of compact sets $K_k\subset \mathbb R^d $ equals  $\mathbb R^d$.

 Set
$$
C_{X^{(h)}}(x,y)=C_{X^{h}}(x-y)=\sum_{\mathbf j\in \mathbb Z^d} C_{X^{(h)}}(h(\mathbf{j})) \prod_{p=1}^d  \frac{\sin( \pi ( h^{-1} (x_p-y_p) -  j_p  ))}{ \pi( h^{-1}(x_p-y_p) - j_p)  }
$$
for all $x,y\in \mathbb R^n$.  Then 
\begin{eqnarray*}
\begin{split}
\widehat C_{X^{(h)}}(\xi) =& \sum_{\mathbf j_\in\mathbb Z}  C_{X^{(h)}} (h\mathbf {j}) \exp( i 
     h  \mathbf {j}\cdot \xi )   h^{d} 1_ {[-\pi ,\pi  ]^d} (h\xi  )\\
 =&   \sigma^2  \bigg( \sum_{k=0}^{K}  a_k \kappa^{2(\alpha- k)}
 \sum_{p =1}^d  h^{-2k}(2- 2\cos( h \xi_p))^k   \bigg)^{-1}  1_ {[-\pi ,\pi  ]^d} (h\xi  )
 \end{split}
\end{eqnarray*}
which  is equivalent to
$$
C_{X^{(h)}} (x)=  \frac{\sigma^2}{(2\pi)^d}  \int_{\left(-\frac{\pi}{h},\frac{\pi}{h}\right)^d}    \frac{\exp( - i    x \cdot \xi )}{ \sum_{k=0}^{K}  a_k \kappa^{2(\alpha- k)}
 \sum_{p =1}^d  h^{-2k}(2- 2\cos(\xi_p))^k   } d\xi.
 $$
 The weak convergence of $X^{(h_n)}$ to $X$ when  $h_n\rightarrow 0$ as $n\rightarrow \infty$ follows from the next result.
  The same result also shows   the weak convergence of random fields when  the truncation parameter grows.
\begin{theorem}
Let  $X_n$ be  a sequence of zero mean Gaussian  random fields on 
$\mathbb R^d$,   whose covariance functions 
$C_n(x,y)=C_n(x-y)$ satisfy 
\begin{equation}\label{covo}
\vert  \widehat C_n(\xi)\vert  \leq \frac{c}{1+|\xi|^{2s}},
\end{equation}
for some constants $c>0$, $s>(d+1)/2$  and for all $n$.   
If $$\lim_{n\rightarrow \infty }  C_n (x,y)= C(x,y)$$ for all $x,y\in \mathbb R^d$,
then  the distributions of $X_n$  converge weakly on $C(\mathbb R^d)$ to the  distribution of Gaussian zero mean random field $X$ whose  covariance function  is $C(x,y)$.
\end{theorem}
\begin{proof}

The  random fields $X_n$   have continuous sample paths by \eqref{covo} and   Theorem 3.4.3 in \cite{Adler1981}. 
Therefore, their distributions are Gaussian measures on $C(\mathbb R)$, whose characteristic functions   $\mathbf E\left( \exp(i t X_n(x)\right)$ converge to
the characteristic function of  the zero mean Gaussian measure  with covariance function
$C(x,y)$. 

Since the probability density of $X_n(0)$ converges to the probability density of $X(0)$, the sequence $\{X_n(0)\}$ is tight. 
Moreover,
\begin{equation}
\begin{split}
\mathbf E\left( |X_n(x)-X_n(y)|^{2P} \right) =& (2 C_n(0)-2 C_n(x-y))^P\\
= &\left(\frac{2}{(2\pi)^d}\int  (1-  \exp(-i(x-y)\cdot \xi) \widehat C(\xi) d\xi\right)^P\\
\leq & \left(\frac{2}{(2\pi)^d}\int  |x-y||\xi|  | \widehat C_n(\xi) |d\xi \right)^P\\
\leq & C |x-y|^P.    
\end{split}
\end{equation}
Choosing $P>d$ allows application of Kolmogorov-Chentsov tightness 
criterion (see Corollary 16.9 in \cite{Kallenberg}).  
Uniform tightness and the convergence of the characteristic functions
imply the weak convergence (see Corollary 3.8.5 in \cite{Bogachev1998}). 
 
\end{proof}

\section{Numerical experiments}
\label{sec:experiments}
In this section we test the numerical accuracy of 2-dimensional Mat\'ern field approximation with $\sigma^2 = 1$, $\kappa = 1$, and $\alpha = \pi$. Figure~\ref{fig:err_2d} shows the maximum absolute approximation error as function of the Taylor series order. The discretisation step was $h = 0.1$. It can be seen that the error first decreases and then starts to increase. This is to be expected, because at the first steps the approximation becomes better in the central part of the integral while the tails still remain quite heavy. However, when the Taylor series order is increased, the tails become thinner  and hence the approximation error increases and approaches the band-limited covariance function.

\begin{figure}[thp]
\begin{center}
  % replace aims_logo.eps by your figure file name
%  \includegraphics[width=0.6\textwidth]{measurement_placeholder}\\
  \includegraphics[width=0.7\textwidth]{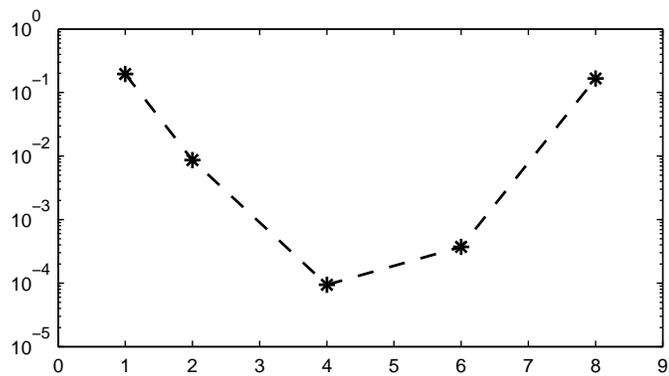}
  \caption{Maximum absolute error in the approximate covariance function with $h = 0.1$ as function of Taylor series order. The minimum error is obtained with order $K = 4$.
  } \label{fig:err_2d}
  \end{center}
\end{figure}

Figure~\ref{fig:approx_2d} shows the approximation with the Taylor series of order $K = 4$ and Figure~\ref{fig:slice_2d} a one-dimensional slice extracted from the middle of the covariance function. It can be seen that the error in the approximation is very small.

\begin{figure}[thp]
\begin{center}
  % replace aims_logo.eps by your figure file name
%  \includegraphics[width=0.6\textwidth]{measurement_placeholder}\\
  \subfigure[Exact]{\includegraphics[width=0.49\textwidth]{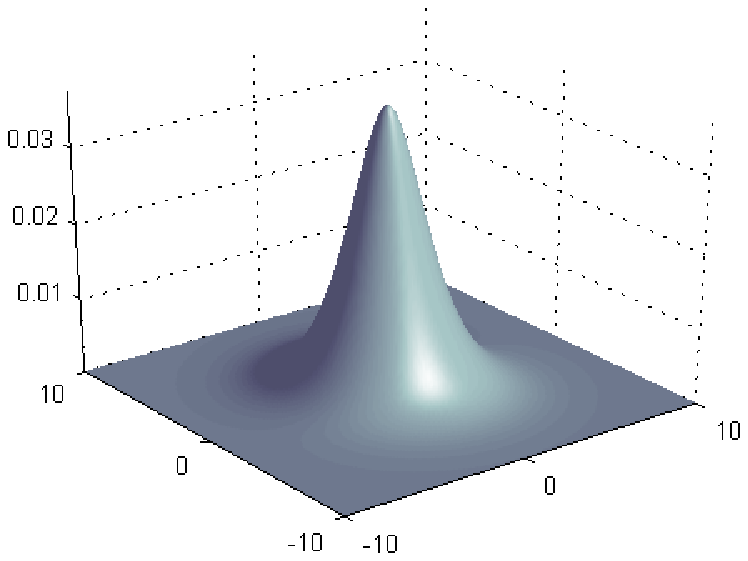}}
  \subfigure[Approximate]{\includegraphics[width=0.49\textwidth]{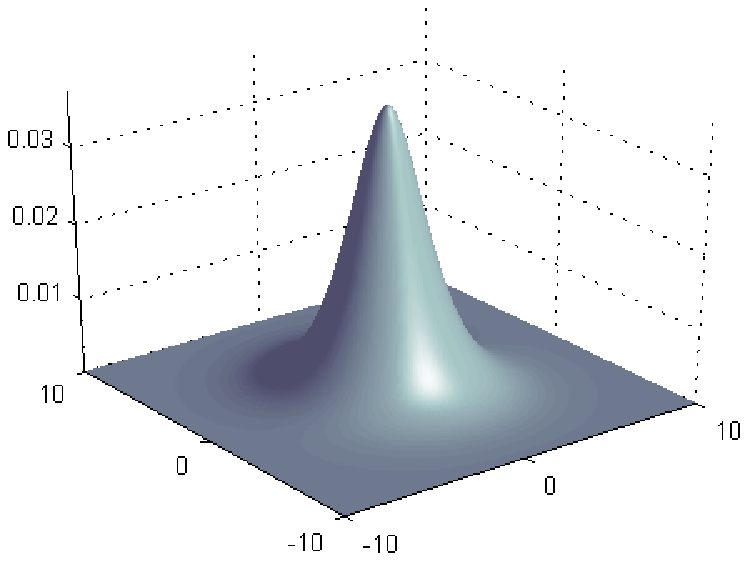}}
  \caption{(a) Exact covariance function. (b) Approximate covariance function with Taylor series of order $K = 4$ and discretisation step $h = 0.1$. The covariance functions are practically indistinguishable.
  } \label{fig:approx_2d}
  \end{center}
\end{figure}

\begin{figure}[htp]
\begin{center}
  % replace aims_logo.eps by your figure file name
%  \includegraphics[width=0.6\textwidth]{measurement_placeholder}\\
  \includegraphics[width=0.7\textwidth]{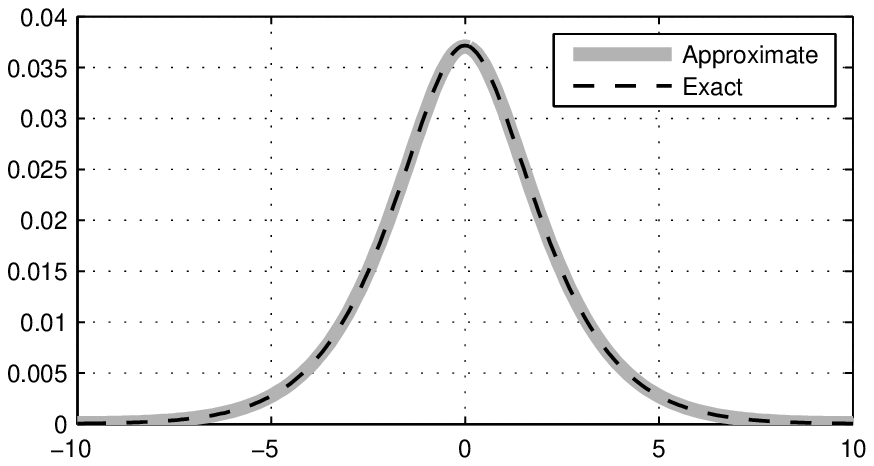}
  \caption{One-dimensional slice from the covariance function in Figure~\ref{fig:approx_2d}. The error in the covariance function is negligible.
  } \label{fig:slice_2d}
  \end{center}
\end{figure}

Finally, Figure~\ref{fig:real_2d} shows an example realisation of the process which is very fast to compute despite the relatively large number of discretisation points (40401).

\begin{figure}[htp]
\begin{center}
  % replace aims_logo.eps by your figure file name
%  \includegraphics[width=0.6\textwidth]{measurement_placeholder}\\
  \includegraphics[width=0.7\textwidth]{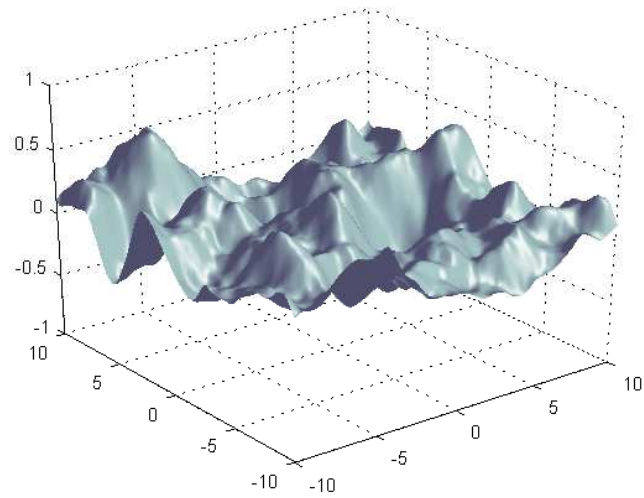}
  \caption{Realisation from the approximate two-dimensional field.
  } \label{fig:real_2d}
  \end{center}
\end{figure}

\section{Conclusion}

We have considered approximation methods of  Mat\'ern fields.
The methodology is based on truncated Taylor expansion for the reciprocal of the power spectrum and Cholesky decomposition for practical computations.
We have shown the convergence of the discrete Mat\'ern field to the continuous ones in two cases, the first one is with respect to Taylor expansion and the second one with the discretisation step $h$.
There are truncation levels that give satisfactory approximations with sparse matrices for Mat\'ern fields.
We have demonstrated corresponding numerical examples.

As our focus was on the methodology, we have not presented any practical applications in this paper.
Hence, this will be one  task of subsequent studies.
Application areas include, for example, tomography within the framework of Bayesian statistical inverse problems or spatial interpolation in Bayesian statistics.

 \section*{Acknowledgements}

This work has been funded by Academy of Finland (project numbers 250215,  266940, and 273475).


\begin{thebibliography}{99}

\bibitem{Adler1981}(MR0611857)
	\newblock R. J. Adler, 
	\newblock ``The geometry of random fields,'' 
	\newblock Wiley Series in Probability and Mathematical Statistics, John Wiley and Sons, Chichester, 1981.
	
\bibitem{Bierman1977}(MR0453090)
	\newblock G. J. Bierman, 
	\newblock ``Factorization Methods for Discrete Sequential Estimation,''
	\newblock Academic Press, New York-London, 1977.

\bibitem{Bogachev1998}(MR1642391)
	\newblock V. I. Bogachev, 
	\newblock ``Gaussian measures,'' (English summary), 
	\newblock Mathematical Surveys and Monographs, American Mathematical Society, Providence, RI, 1998.
	
\bibitem{Davis2011}(MR2865018)
	\newblock T. A. Davis, 
	\newblock \emph{SuiteSparseQR: Multifrontal multithreaded rank-revealing sparse QR factorization},
	\newblock ACM Transactions on Mathematical Software, \textbf{38} (2011), 8:1--8:22. 

\bibitem{Foster2013}(MR3118746)
	\newblock L. V. Foster and T. A. Davis, 
	\newblock \emph{Reliable Calculation of Numerical Rank, Null Space Bases, Pseudoinverse Solutions and Basic Solutions using SuiteSparseQR},
	\newblock ACM Transactions on Mathematical Software, \textbf{40} (2013) 7:1--7:23.
	
\bibitem{Golub1996}(MR1417720)
	\newblock G. H. Golub, C. van Loan, 
	\newblock ``Matrix Computations,'' 3rd Edition, 
	\newblock The Johns Hopkins University Press, 1996.

\bibitem{Grewal2008} 
	\newblock M. S. Grewal and A. P. Andrews, 
	\newblock ``Kalman Filtering: Theory and Practice Using MATLAB, 3rd Edition,''
	\newblock Wiley-IEEE Press, 2008.

\bibitem{Hiltunen2011}(MR2765627)
 	\newblock P. Hiltunen, S. S\"arkk\"a, I. Nissil\"a, A. Lajunen and J. Lampinen,
 	\newblock \emph{State space regularization in the nonstationary inverse problem for diffuse optical tomography},
 	\newblock Inverse Problems, \textbf{27} (2011) 2:025009.

\bibitem{Kaipio2005}(MR2102218) 
	\newblock J.~Kaipio and E.~Somersalo,
	\newblock ``Statistical and Computational Inverse Problems,''
	\newblock Springer-Verlag, 2005.
	
\bibitem{Kallenberg}(MR1876169)
	\newblock O. Kallenberg,
	\newblock ``Foundations of Modern Probability,"
	\newblock Springer-Verlag, 2002.

\bibitem{Lasanen2005}
	\newblock S. Lasanen and L. Roininen, 
	\newblock \emph{Statistical inversion with Green's priors}, 
	\newblock Proc. 5th Int. Conf. on Inv. Prob. in Eng., Cambridge, UK, 11-15th July 2005 L01, 1-10 (2005).
	
\bibitem{Lasanen2012a}(MR2942739)
	\newblock S.~Lasanen, 
	\newblock \emph{Non-Gaussian statistical inverse problems. Part I: Posterior distributions},
	\newblock Inverse Problems and Imaging, \textbf{6} (2012), 215--266.
	
\bibitem{Lindgren2011}(MR2853727)
	\newblock F.~Lindgren, H.~Rue and J.~Lindstr\"om,
	\newblock \emph{An explicit link between Gaussian Markov random fields: the stochastic partial differential equation approach},
	\newblock Journal of the Royal Statistical Society: Series B, \textbf{73} (2011), 423--498.

\bibitem{Lubich1986}(MR0838249)
	\newblock C. Lubich, 
	\newblock \emph{Discretized fractional calculus},
	\newblock SIAM J. Math. Anal. \textbf{17} (1986),  704--719.

\bibitem{Orispaa2010}(MR2671108)
	\newblock M.~Orisp\"a\"a and M.~S.~Lehtinen,
	\newblock \emph{Fortran Linear Inverse Problem Solver},
	\newblock Inverse Problems and Imaging, \textbf{4} (2010) 485-503. 
	
\bibitem{Rasmussen2006}(MR2514435)
	\newblock C.~E.~Rasmussen and C.~K.~I.~Williams,
	\newblock ``Gaussian Processes for Machine Learning,''
	\newblock Adaptive Computation and Machine Learning, 
	\newblock The MIT Press, (2006).

\bibitem{Roininen2011}(MR2773430)
	\newblock  L.~Roininen, M.~Lehtinen, S.~Lasanen, M.~Orisp\"a\"a and M.~Markkanen, 
	\newblock \emph{Correlation priors}, 
	\newblock Inverse Problems and Imaging, \textbf{5} (2011) 167--184.

\bibitem{Roininen2013a}(MR3063550)
	\newblock L.~Roininen, P.~Piiroinen and M.~Lehtinen,
	\newblock \emph{Constructing Continuous Stationary Covariances as Limits of the  Second-Order Stochastic 	Difference Equations},
	\newblock Inverse Problems and Imaging, \textbf{7} (2013) 611--647.

\bibitem{Roininen2014}(MR3209311)
	\newblock L.~Roininen, J.~M.~J.~Huttunen and S.~Lasanen,
	\newblock \emph{Whittle-Mat\'ern priors for Bayesian statistical inversion with applications in electrical impedance tomography},
	\newblock Inverse Problems and Imaging, \textbf{9} (2014) 561--586.

\bibitem{Sarkka2013}
	\newblock S. S\"arkk\"a, A. Solin and J. Hartikainen,
	\newblock \emph{Spatio-Temporal Learning via Infinite-Dimensional Bayesian Filtering and Smoothing},
	\newblock IEEE Signal Processing Magazine, \textbf{30} (2013) 4:51--61.

\bibitem{Simpson2009}
	\newblock D. Simpson, 
	\newblock ``Krylov subspace methods for approximating functions of symmetric positive definite matrices with applications to applied statistics and models of anomalous diffusion,''
	\newblock Ph.D. thesis, Queensland University of Technology, Brisbane, Queensland, Australia 2009.
	
\bibitem{Solin2013}
	\newblock A. Solin and S. S\"arkk\"a (2013),
	\newblock \emph{Infinite-Dimensional Bayesian Filtering for Detection of Quasi-Periodic Phenomena in Spatio-Temporal Data},
 	\newblock Physical Review E, \textbf{88} (2013) 5:052909.

\bibitem{Stuart2010}(MR2652785)
	\newblock A. M. Stuart, 
	\newblock \emph{Inverse problems: a Bayesian perspective}, 
	\newblock Acta Numerica \textbf{19} (2010) 451--559.
	
	

\end{thebibliography}
\end{document}